\documentclass[12pt]{article}

\usepackage[top=25truemm,bottom=20truemm,left=20truemm,right=20truemm]{geometry}
\usepackage[noadjust]{cite}
\usepackage[dvipsnames]{xcolor}
\usepackage{enumitem}
\usepackage[pdftex]{graphicx}
\usepackage{amsmath,amssymb,amsfonts,amssymb,amscd} 
\usepackage{amsthm} 
\usepackage{thmtools}
\usepackage[all]{xy} 
\usepackage{mathtools}
\usepackage{cleveref} 

\theoremstyle{plain}
\newtheorem{thm}{Theorem}[section]
\newtheorem{prop}[thm]{Proposition}
\newtheorem{cor}[thm]{Corollary}

\newtheorem{prob}[thm]{Problem}

\theoremstyle{definition}
\newtheorem{defn}[thm]{Definition}
\newtheorem{example}[thm]{Example}
\newtheorem{rem}[thm]{Remark}

\crefname{thm}{Theorem}{Theorems}
\crefname{cor}{Corollary}{Corollarys}
\crefname{lem}{Lemma}{Lemmas}
\crefname{prop}{Proposition}{Propositions}
\crefname{defn}{Definition}{Definitions}
\crefname{example}{Example}{Examples}
\crefname{rem}{Remark}{Remarks}
\crefname{figure}{Figure}{Figures}
\crefname{section}{Section}{Section}

\newcommand{\qop}{\triangleleft}
\newcommand{\HH}{\mathbb{H}}
\newcommand{\RR}{\mathbb{R}}

\newcommand{\GAlex}{\mathrm{GAlex}}

\title{Detecting non-admissibility of quandles via colorings}
\author{Katsunori Arai\footnote{arai1223math@gmail.com} and Ryoya Kai \footnote{kai.ryoya.d8@cc.nara-edu.ac.jp}}
\date{\ }

\begin{document}

\maketitle


\begin{abstract}
    A quandle is an algebraic system whose axioms are motivated by Reidemeister moves in knot theory.
    A typical example is a conjugation quandle 
    arising from a group.
    A quandle is said to be admissible if it is isomorphic to a conjugation quandle.
    Admissible quandles often yield knot invariants that coincide with those derived from the knot group, whereas non-admissible quandles may produce genuinely new invariants.
    In this sense,
    it is important to construct non-admissible quandles.
    In this paper, 
    we provide criteria for determining whether given quandles are admissible by colorings of $(1,1)$-tangles. 
    As an application, we construct numerous examples of non-admissible quandles by analyzing simple tangles obtained from the Hopf link and the trefoil knot.
\end{abstract}

\section{Introduction}

Quandles were introduced as algebraic structures for constructing link
invariants; their axioms correspond to Reidemeister moves in knot theory. 
Beyond knot theory, they can be regarded as a generalization of conjugation in groups and are closely related to the algebraic properties of point symmetries of symmetric spaces.
These connections have led to developments and applications in various areas of mathematics. 
In particular, the fundamental quandle of an oriented link is a remarkably strong link invariant. 

However, the algebraic structure of the fundamental quandle of an oriented link is often too complicated to analyze directly. 
Therefore, it is common to study quandle homomorphisms from the fundamental quandle to a fixed, explicitly described quandle. 
Such homomorphisms can be described combinatorially as assignments of elements of the quandle to the arcs of an oriented link diagram satisfying certain conditions. 
These assignments are called colorings of the diagram by the quandle, and are in bijection with homomorphisms from the fundamental quandle to the quandle. 
Moreover, using the functor that associates a group to a quandle, 
quandle colorings can be related to group homomorphisms of the knot group into an appropriate group. 
From this perspective, quandle colorings refine the notion of group homomorphisms from knot groups.

Every group gives rise to a quandle via conjugation, and its subquandles are called conjugation quandles. 
A quandle is said to be admissible if it is isomorphic to a conjugation quandle. 
In this case, quandle homomorphisms from the fundamental quandle correspond to group homomorphisms from the knot group. 
Consequently, to extract genuinely new information beyond invariants derived from the knot group, 
it is essential to consider non-admissible quandles. This motivates the following problem:

\begin{prob}[cf. {\cite[Question 3.1]{Bardakov-2017-AutomorphismGroupsQuandlesArising}}]
  Determine whether a given quandle is admissible or not.
\end{prob}
In general, it is difficult to determine whether a given quandle is admissible.
In fact, to solve this problem, we need to solve the word problem for certain elements in the associated group of the quandle.

In this paper, we introduce the notion of $(L,p)$-admissibility for quandles, where $L$ is an oriented link, and $p$ is a base point on $L$. This notion is described in terms of colorings of a diagram of the associated $(1,1)$-tangle, which is obtained by cutting $L$ at $p$.
As an application, we provide explicit examples of non-admissible quandles by analyzing simple $(1,1)$-tangles obtained from the Hopf link and the trefoil knot.

\section{Preliminaries}
In this section, we review basic definitions and examples of quandles used throughout the paper.

\begin{defn}
  Let $X$ be a nonempty set and $\qop: X \times X \to X$ a binary operation on $X$.
  The pair $(X, \qop)$ is called a \textit{quandle} if it satisfies the following conditions:
  \begin{itemize}
    \item For any $x \in X$, we have $x \qop x = x$.
    \item For any $y \in X$, the map $S_{y}: X \to X$, defined by $S_{y}(x) = x \qop y$, is bijective.
    \item For any $x, y, z \in X$, 
    the identity $(x \qop y) \qop z = (x \qop z) \qop (y \qop z)$ holds.
  \end{itemize}
\end{defn}
For $x, y \in X$, we denote $S_{y}^{-1}(x)$ by $x \overline{\qop} y$.
For simplicity, we often write $X$ for the quandle $(X, \qop)$
when no confusion can occur.

Let $X$ and $Y$ be quandles.
A map $f: X \to Y$ is a \textit{quandle homomorphism}
if it preserves the quandle operations.
A \textit{quandle isomorphism} is a bijective quandle homomorphism.
Two quandles are said to be \textit{isomorphic} 
if there exists a quandle isomorphism from one to the other.

A \textit{subquandle} of a quandle $(X, \qop)$ is a subset $A \subset X$ that is closed under the operations $\qop$ and $\overline{\qop}$.
If $A$ is a subquandle of $X$, then the inclusion map $i: A \hookrightarrow X$ is an injective quandle homomorphism.
More generally, for a quandle homomorphism $f: X \to Y$, the image $f(X)$ is a subquandle of $Y$.
In addition, if the homomorphism $f$ is injective, then the image $f(X)$ is isomorphic to $X$.

\begin{example}\label{Ex:conjugation_quandle}
  A group $G$ becomes a quandle with the operation $g \qop h := h^{-1}gh$ for $g, h \in G$, denoted by $\mathrm{Conj}(G)$.
  A subquandle of $\mathrm{Conj}(G)$ is called a \textit{conjugation quandle}. We note that any subquandle of $\mathrm{Conj}(G)$ consists of some conjugacy classes of a subgroup of the group $G$.
\end{example}

\begin{example}\label{Ex:GAlex}
  Given a group $G$ and an automorphism $\sigma \in \mathrm{Aut}(G)$,
  define the binary operation $g \qop h := \sigma (g h^{-1}) h$ for $g, h \in G$.
  Then $G$ becomes a quandle.
  This quandle is denoted by $\mathrm{GAlex}(G, \sigma)$, and is called a generalized Alexander quandle.
\end{example}

\begin{defn}
  A quandle $X$ is said to be \textit{admissible} 
  if there exists an injective quandle homomorphism $X \to \mathrm{Conj}(G)$ for a group $G$.
  Otherwise, $X$ is said to be \textit{non-admissible}.
  In other words,
  an admissible quandle is isomorphic to a subquandle of $\mathrm{Conj}(G)$ for some group $G$.
\end{defn}

\begin{rem}
  The notion of admissible quandles 
  was introduced in \cite{Kamada-2005-EnvelopingMonoidalQuandles}. 
  An admissible quandle is also referred to as reduced \cite{Ryder-1996-AlgebraicConditionDetermineWhethera,Inoue-2010-KnotQuandlesInfiniteCyclica}, 
  or embeddable \cite{Akita-2023-EmbeddingAlexanderQuandlesGroupsa}.
\end{rem}

\begin{rem}
  It is known that several classes of quandles are admissible,
  including Alexander quandles, core quandles, and twisted conjugation quandles.
  For further details, see the introduction of \cite{Akita-2023-EmbeddingAlexanderQuandlesGroupsa}.
\end{rem}
\section{
  $(L,p)$-admissibility for quandles
}

In this section, 
we introduce a criterion for detecting non-admissibility of quandles.
The key idea is to study colorings of a $(1,1)$-tangle obtained from an oriented link with a base point.

A \textit{$(1,1)$-tangle} is the image of a smooth embedding $f: \mathbb{R} \sqcup S^{1} \sqcup \cdots \sqcup S^{1} \hookrightarrow \mathbb{R}^{3}$ with the following condition:
there exist $a, b \in \mathbb{R}$ such that $f|_{\mathbb{R}}(t) = (0,0,t)$ for any $t \in \mathbb{R} \setminus (a, b)$.
Throughout this paper,
we assume that every $(1,1)$-tangle is oriented.
In addition, the component $f(\mathbb{R})$
is equipped with the orientation directed from $f(a)$ to $f(b)$.
A \textit{diagram} of a $(1,1)$-tangle is defined as usual in knot theory. 
A diagram of an oriented $(1,1)$-tangle is obtained from a diagram $D$ of an oriented link $L$ with a base point $p \in L$ as follows:
we assume that the point of $D$ corresponding to $p$ lies in the interior of an arc of $D$.
By cutting the diagram $D$ at the base point,
we obtain the diagram $\tilde{D}$ of the $(1,1)$-tangle $\tilde{L}$, as illustrated in \cref{Fig:1_1-tangle}.
Note that if two base points $p, q \in L$ are contained in the same component, 
then we obtain equivalent $(1,1)$-tangles; that is, they are ambient isotopic in $\RR$ relative to $\RR^{2} \times (\RR \setminus (a, b)$ for some $a, b \in \RR$.
The $(1,1)$-tangle $\tilde{L}$ is called the \textit{$(1,1)$-tangle obtained from the pair $(L, p)$ of the link $L$ and the base point $p$}.
In what follows, we denote by $a_{s}$ and $a_{t}$ the arcs containing $f(a)$ and $f(b)$, respectively.

\begin{figure}[h]
  \centering
  \includegraphics{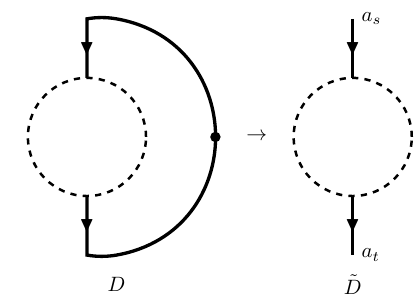}
  \caption{A diagram $\tilde{D}$ of the $(1,1)$-tangle $\tilde{L}$ obtained from $D$}
  \label{Fig:1_1-tangle}
\end{figure}

We now recall the notion of a quandle coloring of a diagram.
\begin{defn}
  Let $D$ be a diagram of an oriented link $L$ or a $(1,1)$-tangle $\tilde{L}$,
  and let $\mathrm{Arc}(D)$ denote the set of all arcs of $D$.
  For a quandle $X$, a map $C: \mathrm{Arc}(D) \to X$ is called an \textit{$X$-coloring} of $D$ 
  if it satisfies the following condition at each crossing:
  the map $C$ satisfies $C(a_{i}) \triangleleft C(a_{j}) = C(a_{k})$, 
  where $a_{i}, a_{j}, a_{k} \in \mathrm{Arc}(D)$ 
  are as shown in \cref{Fig:Coloring_condition}.
  \begin{figure}
    \centering
    \includegraphics{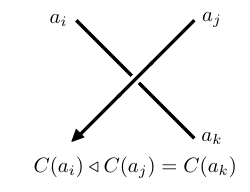}
    \caption{A coloring condition}
    \label{Fig:Coloring_condition}
  \end{figure}
\end{defn}

Using quandle colorings of a diagram of $(1,1)$-tangle, we introduce a notion of admissibility for quandles.

\begin{defn}
  Let $L$ be an oriented link with a base point $p$,
  and let $\tilde{L}$ be the $(1,1)$-tangle obtained from $(L, p)$.
  We say that a quandle $X$ is \textit{$(L,p)$-admissible} 
  if any $X$-coloring $C$ of a diagram $D$ of $\tilde{L}$ satisfies $C(a_{s})=C(a_{t})$.
\end{defn}

\begin{rem}
  The $(L,p)$-admissibility generally depends on the choices of the orientation of $L$ and of the base point $p$.
  In particular, if $K$ is an oriented knot and $p \in K$ is a base point, then the $(1,1)$-tangle $\tilde{K}$ obtained from $(K,p)$ is independent of the choice of base point on $K$.
  A quandle $X$ is said to be \textit{$L$-admissible} if it is $(L, p)$-admissible for every base point $p \in L$.
  Hence, for an oriented knot $K$, we simply say that a quandle is \textit{$K$-admissible}.
\end{rem}

The following proposition shows that $(L,p)$-admissibility is a necessary condition for admissibility.
This can be viewed as a refinement of \cite{Nosaka-2011,Clark-2016}.

\begin{prop}
  Any admissible quandle is $(L,p)$-admissible for any link $L$ with a base point $p \in L$.
  \label{Prop:K-admissible_condition}
\end{prop}

A proof of \cref{Prop:K-admissible_condition} is given in \cref{Sec:Proof}.
As a consequence, we obtain the following criterion:
\begin{cor}\label{Cor:Non-admissible_condition}
  If a quandle is not $(L, p)$-admissible for some link $L$ with a base point $p$,
  then the quandle is non-admissible.
\end{cor}
\section{Non-admissible quandles}

In this section, we determine 
whether some quandles are non-admissible 
using the criterion given in \cref{Cor:Non-admissible_condition}.
To apply the criterion, we focus on two of the simplest non-trivial links, namely, the Hopf link $2_{1}^{2}$ and the trefoil knot $3_{1}$.
They are the non-trivial link and the non-trivial knot with the smallest number of crossings, respectively.

\subsection{The $(1,1)$-tangle obtained from the Hopf link}

In this subsection, we investigate admissibility with respect to the Hopf link $2_1^2$.
First, we give the condition arising from the $(1, 1)$-tangle obtained from Hopf link $2_{1}^{2}$. 

\begin{prop}
  A quandle $X$ is not $2_{1}^{2}$-admissible if and only if there exist $x, y \in X$ such that $x \qop y = x$ and $y \qop x \neq y$.
  \label{Prop:HopflinkNAdm}
\end{prop}
\begin{proof}
  Let $D$ be the diagram of the $(1,1)$-tangle obtained from the Hopf link $2_{1}^{2}$ as shown in \cref{Fig:hopftanglediagram}.
  Then, any $X$-coloring of $D$ is determined as follows:
  first,
  we assign two elements $x, y \in X$ to the arcs as in \cref{Fig:hopftanglediagram}.
  Then, the element assigned to the arc $a_{t}$ is determined by the relation at the crossing $c_{1}$.
  Finally, if the elements $x$ and $y$ satisfy the relation at $c_{2}$, which is given by
  \begin{equation*}
    x \qop (y \qop x) = x \Leftrightarrow x \qop y = x,
  \end{equation*}
  then the map $C: \mathrm{Arc}(D) \to X$ is an $X$-coloring.
  In addition,
  it satisfies
  \begin{equation*}
    C(a_{s}) \neq C(a_{t}) \Leftrightarrow y \qop x \neq y.
  \end{equation*}
  We note that the above condition does not depend on the choice of an orientation of $2_{1}^{2}$ and a base point.
  Therefore,
  the quandle $X$ satisfies the condition of this proposition if and only if $X$ is not $2_{1}^{2}$-admissible.  
  \begin{figure}[h]
    \centering
    \includegraphics{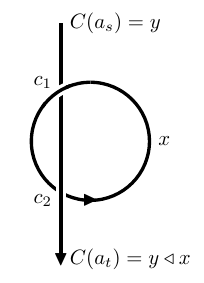}
    \caption{An $X$-colored $(1,1)$-tangle diagram $D$ obtained from a diagram of the Hopf link $2_{1}^{2}$}
    \label{Fig:hopftanglediagram}
  \end{figure}
\end{proof}

The following is immediately obtained from \cref{Cor:Non-admissible_condition}.

\begin{cor}
  Let $X$ be a quandle.
  If there exist $x, y \in X$ such that $x \qop y = x$ and $y \qop x \neq y$, then $X$ is non-admissible.
  \label{Prop:ConditionbyHopflink}
\end{cor}

\begin{rem}
  The contrapositive of \cref{Prop:ConditionbyHopflink} also follows from a group-theoretic calculation.
  Indeed, for elements $x,y$ in a group, the identity $y^{-1}xy = x$
  holds if and only if $x^{-1}yx = y$.
  Thus, in any conjugation quandle, if $x \qop y = x$, then $y \qop x = y$.
  Since every admissible quandle is isomorphic to a conjugation quandle, the same implication holds in every admissible quandle.
\end{rem}

As an application of \cref{Prop:HopflinkNAdm}, we construct many non-admissible quandles.

\begin{thm}
  Let $G$ be a group and
  $N \subset G$ a normal subgroup of $G$.
  We define the binary operation $\qop$ on $G \times N$ by 
  \[(g_{1}, n_{1}) \qop (g_{2}, n_{2}) = ((g_{2} n_{2})^{-1} (g_{1} n_{1}) g_{1} (g_{1} n_{1})^{-1} (g_{2} n_{2}), (g_{2} n_{2})^{-1} (g_{1} n_{1}) n_{1} (g_{1} n_{1})^{-1} (g_{2} n_{2})).\]
  Then the following hold:
  \begin{itemize}
    \item[(i)] $X = (G \times N, \qop)$ is a quandle.
    \item[(ii)] If $N \not\subset Z(G)$, then the quandle $X$ is non-admissible.
  \end{itemize}
  \label{Thm:Hopflink}
\end{thm}
\begin{proof}
  (i) This is verified by direct calculation.

  (ii) Assume $N \not\subset Z(G)$.
  Then there exist $A \in G$ and $B \in N$ such that $AB \neq BA$.
  Set $x = (e,e), y = (A, B) \in X$.
  Then
  \begin{align*}
    x \qop y & = (e,e) \qop (A, B) \\
    & = (e,e),\\
    y \qop x & = (A, B) \qop (e,e) \\
    & = ((AB)A(AB)^{-1}, (AB)B(AB)^{-1})\\
    & = ((AB)A(AB)^{-1}, ABA^{-1}).
  \end{align*}
  Since $AB \neq BA$, we have $ABA^{-1} \neq B$.
  Hence, $y \qop x \neq y.$
  Using Proposition~\ref{Prop:ConditionbyHopflink},
  the quandle $X$ is non-admissible.
\end{proof}

Unfortunately, \cref{Prop:ConditionbyHopflink} 
cannot be used to detect non-admissibility of
generalized Alexander quandles.

\begin{prop}
    Every generalized Alexander quandle is $2_{1}^{2}$-admissible.
    \label{Prop:Hopf_vs_GAlex}
\end{prop}
\begin{proof}
    Let $X := \GAlex(G, \sigma)$ be a generalized Alexander quandle.
    Let us take $x, y \in X$.
    By \cref{Prop:HopflinkNAdm},
    it is enough to show that if $x \qop y = x$, then $y \qop x = y$.
    Here, we have 
    \[
        x = x \qop y = \sigma(xy^{-1})y.
    \]
    and hence $x y^{-1} = \sigma(xy^{-1})$.
    Thus, 
    \[
        y \qop x = \sigma(yx^{-1}) x = (\sigma(xy^{-1}))^{-1} x = (xy^{-1})^{-1}x = y,
    \]
    as desired.
\end{proof}

\subsection{The $(1,1)$-tangle obtained from the trefoil knot}

The trefoil knot $3_1$ is the nontrivial knot with the smallest crossing number.
In this subsection, we give some examples of 
quandles that are not $3_1$-admissible. 
First, we give the condition with respect to the $(1, 1)$-tangle obtained from the trefoil knot.

\begin{prop}\label{Prop:31-admissible}
  A quandle $X$ is not $3_1$-admissible
  if and only if 
  there exist elements $x, y \in X$ such that $(x \triangleleft y) \triangleleft x = y$ and $(y \triangleleft x) \triangleleft y \neq x$.
\end{prop}

\begin{proof}
  This is proved similarly to the proof of \cref{Prop:HopflinkNAdm}.
  Since the trefoil knot $3_1$ is invertible, the conditions do not depend on the choice of an orientation of $3_{1}$.
  We note that the first equality is given by the relation corresponding to the crossing $c_3$, 
  and the second inequality is given by the elements assigned to the arcs $a_s$ and $a_t$.
\end{proof}

By \cref{Prop:31-admissible} and \cref{Cor:Non-admissible_condition},
the following corollary follows.

\begin{cor}\label{Prop:ConditionbyTrefoil}
  Let $X$ be a quandle.
  If there exist $x, y \in X$ such that $(x \triangleleft y) \triangleleft x = y$ and $(y \triangleleft x) \triangleleft y \neq x$, then $X$ is non-admissible.
\end{cor}

\begin{figure}[ht]
    \centering
    \includegraphics{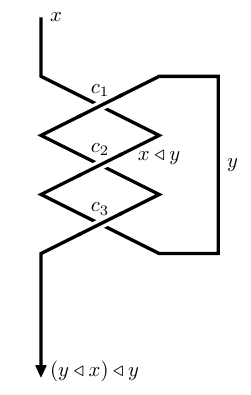}
    \caption{An $X$-colored diagram $D$ of the $(1,1)$-tangle obtained from  $3_{1}$}
    \label{Fig:Trefoiltanglediagram}
\end{figure}

As noted in \cref{Prop:Hopf_vs_GAlex}, 
every generalized Alexander quandle is $2_1^2$-admissible.
In other words, the $(1, 1)$-tangle obtained from the Hopf link $2_{1}^{2}$ does not yield a useful criterion for detecting non-admissibility among generalized Alexander quandles.
In contrast, the $(1, 1)$-tangle obtained from the trefoil knot can yield such a criterion.
Let $Q_8 = \{\pm 1, \pm i, \pm j, \pm k\}$ be the quaternion group, and let $\sigma: Q_8 \to Q_8$ be the group automorphism defined by 
\[
  \sigma(1) = 1, \quad \sigma(i) = j, \quad \sigma (j) = k, \quad \sigma(k) = i.
\]
Then, we obtain the generalized Alexander quandle $\mathrm{GAlex}(Q_8, \sigma)$.
\begin{prop}\label{prop:GAlexQ8}
  The quandle $\GAlex(Q_8, \sigma)$ is not $3_1$-admissible.
\end{prop}
\begin{proof}
  Take $x = 1$ and $y = i$. We can verify the conditions in \cref{Prop:ConditionbyTrefoil}, and hence the quandle $\GAlex(Q_8, \sigma)$ is not $3_1$-admissible.
\end{proof}

Let $\HH$ be the set of quaternions, that is, $\HH$ be the real vector space spanned by $\{1, i, j, k\}$.
The norm of $q = a + bi + cj + dk \in \HH$ for $a,b,c,d \in \RR$ is given by 
\[
  ||q|| := \sqrt{a^2+b^2+c^2+d^2}.
\]
The quaternionic unitary group $Sp(1)$ is the simple Lie group consisting of linear isomorphisms on $\HH$ which preserves the norm.
We note that the group $Q_8$ is naturally embedded in $Sp(1)$, and the group automorphism $\sigma: Q_8 \to Q_8$ extends to the group isomorphism $\widetilde{\sigma}: Sp(1) \to Sp(1)$.
Thus, the generalized Alexander quandle $\GAlex(Sp(1), \widetilde{\sigma})$ is smooth.
In addition,
the quandle $\GAlex(Sp(1), \widetilde{\sigma})$ is algebraically and topologically connected.
Since the quandle $\GAlex(Q_8, \sigma)$ is a subquandle of $\GAlex(Sp(1), \widetilde{\sigma})$, we have the following corollary.
Note that this provides a counterexample to Yonemura's conjecture \cite{Yonemura2024}.

\begin{cor}
  The quandle $\GAlex(Sp(1), \widetilde{\sigma})$ is not $3_1$-admissible. 
\end{cor}

The rest of this section presents the result of a computer calculation. 
Higashitani--Kamada--Kosaka--Kurihara \cite{Higashitani-2024-ClassificationGeneralizedAlexanderQuandlesa} gave an algorithm to determine whether given two generalized Alexander quandles are isomorphic or not.
By applying the algorithm, they classified finite generalized Alexander quandles of order less than $128$, 
and found $8726$ quandles up to isomorphism.
We checked the condition in \cref{Prop:ConditionbyTrefoil}  using a computer, and we obtained the following result.

\begin{prop}
  There are $287$ quandles of order less than $128$ that are not $3_1$-admissible.
\end{prop}
\section{Proof of \cref{Prop:K-admissible_condition}}
\label{Sec:Proof}

\subsection{Algebraic condition}
First, we recall the condition of admissibility for quandles.
The condition relates to the associated groups of quandles.

\begin{defn}
  Let $X = (X, \qop)$ be a quandle.
  The \textit{associated group} of $X$, 
  denoted by $\mathrm{As}(X)$,
  is the group defined by the following group presentation:
  \begin{equation*}
    \mathrm{As}(X) = \left\langle x \in X \mid x \qop y = y^{-1} x y\ (x, y \in X)\right\rangle. 
  \end{equation*}
\end{defn}

The \textit{natural map} $\eta_{X}: X \to \mathrm{As}(X)$ is the composition of the inclusion map $\iota: X \to F_{\text{grp}}(X)$ and the quotient map $\pi: F_{\text{grp}}(X) \to \mathrm{As}(X)$, i.e.,
$\eta_{X} = \pi \circ \iota$.
We note that, when regarded as the map $\eta_{X}: X \to \mathrm{Conj}(\mathrm{As}(X))$, the map $\eta_{X}$ is a quandle homomorphism.

\begin{rem}
The assignment that associates to a quandle with its associated group defines a functor from the category of quandles to the category of groups \cite{Joyce-1982-ClassifyingInvariantKnotsKnotb}.
\label{Rem:AsFunctor}
\end{rem}

\begin{prop}[\cite{Joyce-1982-ClassifyingInvariantKnotsKnotb}]
Let $X$ be a quandle and $G$ a group.
For any quandle homomorphism $f: X \to \mathrm{Conj}(G)$,
there exists a unique group homomorphism $f_{\ast}: \mathrm{As}(X) \to G$ such that the following diagram is commutative:
\[
  \begin{CD}
        X @>{\eta_{X}}>> \mathrm{As}(X) \\
    @V{f}VV             @VV{f_{\ast}}V \\
    \mathrm{Conj}(G) @>{\mathrm{id}}>> G
  \end{CD}
\]
\label{Prop:Uni_property_of_As}
\end{prop}
See \cite{Kamada-2017-SurfaceKnots4Spacea} for a proof of \cref{Prop:Uni_property_of_As}.

\begin{prop}[\cite{Joyce-1982-ClassifyingInvariantKnotsKnotb,Kamada-2005-EnvelopingMonoidalQuandles}]
  Let $X = (X, \triangleleft)$ be a quandle.
  The natural map $\eta_{X}: X \to \mathrm{As}(X)$ is injective if and only if $X$ is admissible.
  \label{Prop:Admissible}
\end{prop}
\begin{proof}
  If the natural map $\eta_{X}: X \to \mathrm{As}(X)$ is injective, then the map $\eta_{X}: X \to \mathrm{Conj}(\mathrm{As}(X))$ is an injective quandle homomorphism. 
  Hence, the quandle $X$ is admissible.

  Conversely,
  assume that the quandle $X$ is admissible.
  Then, there exist a group $G$ and an injective quandle homomorphism $f: X \to \mathrm{Conj}(G)$.
  Using \cref{Prop:Uni_property_of_As},
  there exists a unique group homomorphism $f_{\ast}: \mathrm{As}(X) \to G$ such that $f = f_{\ast} \circ \eta_{X}$.
  Since $f$ is injective, the map $f_{\ast} \circ \eta_X$ is also injective.
  Hence, the natural map $\eta_{X}$ is injective. 
\end{proof}

\subsection{Topological condition}
Let $L$ be an oriented link or a $(1,1)$-tangle, and let $D$ be a diagram of $L$.
The \textit{fundamental quandle} of $D$ is the quandle generated by $\mathrm{Arc}(D)$, the set of all arcs of $D$, 
with the relations $\{R(c_i)\}$ corresponding to each crossing $c_i$ of $D$:
\begin{equation*}
  Q(D) = \left\langle \mathrm{Arc}(D) \mid R(c_{1}), \ldots, R(c_{m}) \right\rangle_{\text{qdle}},
\end{equation*}
where the relations are defined as follows:
let $c$ be a crossing of $D$,
and let $a_{i}, a_{j}$, and $a_{k}$ be arcs as illustrated in \cref{Fig:relation}.
Then, the relation $R(c)$ is defined by $a_{i} \qop a_{j} = a_{k}$.
\begin{figure}
  \centering
  \includegraphics{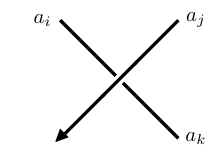}
  \caption{Arcs around the crossing $c$}
  \label{Fig:relation}
\end{figure}
See~\cite{Kamada-2017-SurfaceKnots4Spacea} for details on quandle presentations.
Since $Q(D)$ is independent of the choice of diagram of $L$ up to isomorphism, we simply write $Q(L)$ for $L$.

\begin{prop}[cf.~\cite{Kamada-2017-SurfaceKnots4Spacea}]
  Let $L$ be an oriented link or a $(1,1)$-tangle.
  For a quandle $X$, any $X$-coloring $C$ of a diagram $D$ of $L$ uniquely extends to a quandle homomorphism $C: Q(L) \to X$.
  In addition, any quandle homomorphism $Q(L) \to X$ is obtained by extending an $X$-coloring of $D$.
  \label{Prop:Qdle_coloring}
\end{prop}

In what follows,
we identify an $X$-coloring with the corresponding quandle homomorphism $Q(L) \to X$.

\begin{rem}
Note that a quandle $X$ is $(L,p)$-admissible if and only if any $X$-coloring of the $(1,1)$-tangle $\tilde{L}$ obtained from $(L,p)$ factors through the fundamental quandle $Q(L)$.
That is, for any quandle homomorphism $\tilde{C}: Q(\tilde{L}) \to X$, there exists a quandle homomorphism $C: Q(L) \to X$ such that the following diagram commutes:
\[
\xymatrix{
  Q(\tilde{L}) \ar[d] \ar[rd]^{\tilde{C}} & \\
  Q(L) \ar[r]^-C & X
}\]
\end{rem}

\begin{proof}[Proof of \cref{Prop:K-admissible_condition}]
  Suppose that a quandle $X$ is admissible.
  Let $\tilde{L}$ be the $(1,1)$-tangle obtained from an oriented link $L$ with a base point $p \in L$, and let $D$ be its diagram.
  We now show that the quandle $X$ is $(L, p)$-admissible, that is, any $X$-coloring $C$ of $D$ satisfies $C(a_s) = C(a_t)$.

  By the functoriality of assigning the associated group \cite{Joyce-1982-ClassifyingInvariantKnotsKnotb}, 
  we have the following commutative diagram:
  \[
  \xymatrix{
    Q(\tilde{L}) \ar[d]^{\eta_{Q(\tilde{L})}} \ar[r]^{C} & X \ar[d]^{\eta_{X}}\\
    \mathrm{As}(Q(\tilde{L})) \ar[r]^{\mathrm{As}(C)} & \mathrm{As}(X)
}\]
Furthermore, according to \cite{Joyce-1982-ClassifyingInvariantKnotsKnotb},
the group $\mathrm{As}(Q(\tilde{L}))$ is isomorphic to the link group $G(L)$ and we have $\eta_{Q(\tilde{L})}(a_{s}) = \eta_{Q(\tilde{L})}(a_{t})$. 
Hence,
\begin{align*}
  \eta_{X} \circ C (a_{s})
  = \mathrm{As}(C) \circ \eta_{Q(\tilde{L})} (a_{s}) = \mathrm{As} (C) \circ \eta_{Q(\tilde{L})} (a_{t}) 
  = \eta_{X} \circ C (a_{t}).
\end{align*}
Since the quandle $X$ is admissible,
the natural map $\eta_{X}$ is injective.
Thus, 
we conclude that $C(a_{s}) = C(a_{t})$.
Therefore, the quandle $X$ is $(L, p)$-admissible. 
\end{proof}

\section*{Acknowledgment}
  The authors would like to thank Seiichi~Kamada and Hiroshi~Tamaru for discussions on this research.
  The authors also would like to thank Keisuke~Himeno, Yuko~Ozawa, and Yuta~Taniguchi for constant encouragements.
  The first author was supported by JST SPRING, Grant number JPMJSP2138.
  The second author was supported by JST SPRING, Grant Number JPMJSP2139.

\bibliography{ConjCond}
\bibliographystyle{abbrv}
\end{document}